\documentclass[letterpaper, 10 pt, conference]{ieeeconf}  % Comment this line out if you need a4paper

\IEEEoverridecommandlockouts                              % This command is only needed if 
\usepackage{amsmath,amssymb,bm,bbm,mathrsfs,amscd}
\usepackage{amssymb}
\usepackage{algorithmic}
\usepackage[ruled,vlined]{algorithm2e}
\usepackage{xcolor}
\usepackage{graphicx}
\usepackage{lipsum}
\usepackage{subcaption}
\usepackage{multicol}

\newtheorem{theorem}{Theorem}
\newtheorem{lemma}{Lemma}
\newtheorem{assumption}{Assumption}

\newcommand{\R}{\mathbb{R}}
\newcommand{\N}{\mathbb{N}}
\newcommand{\Ex}{\mathbb{E}}
\newcommand{\T}{\mathbb{T}}
\newcommand{\Y}{\mathbb{Y}}
\newcommand{\X}{\mathbb{X}}

\newcommand{\bP}{\mathbf{P}}
\newcommand{\bQ}{\mathbf{Q}}

\newcommand{\eqdef}{\triangleq}

\newcommand{\bK}{\mathbf{K}}
\newcommand{\bS}{\mathbf{S}}
\newcommand{\bR}{\mathbf{R}}
\newcommand{\bU}{\mathbf{U}}
\newcommand{\Rn}{\R^n}

\newcommand{\scrA}{\mathcal{A}}
\newcommand{\scrB}{\mathcal{B}}
\newcommand{\scrC}{\mathcal{C}}

\newcommand{\scrE}{\mathcal{E}}

\newcommand{\scrW}{\mathcal{W}}
\newcommand{\scrF}{\mathcal{F}}
\newcommand{\scrH}{\mathcal{H}}
\newcommand{\scrG}{\mathcal{G}}
\newcommand{\scrN}{\mathcal{N}}
\newcommand{\scrP}{\mathcal{P}}
\newcommand{\scrX}{\mathcal{X}}
\newcommand{\scrZ}{\mathcal{Z}}
\newcommand{\cm}{\mathtt{m}}
\newcommand{\cM}{\mathtt{M}}

\newcommand{\setA}{\mathsf{A}}
\newcommand{\setC}{\mathsf{C}}
\newcommand{\setG}{\mathsf{G}}
\newcommand{\setM}{\mathsf{M}}
\newcommand{\setN}{\mathsf{N}}
\newcommand{\setK}{\mathsf{K}}

\renewcommand{\Pr}{\mathbb{P}}
\newcommand{\inner}[1]{\langle#1\rangle}
\newcommand{\norm}[1]{\|#1\|}

\DeclareMathOperator{\rint}{relint}
\DeclareMathOperator*{\argmin}{argmin}
\DeclareMathOperator{\prox}{Prox}
\DeclareMathOperator{\dom}{dom}
\DeclareMathOperator{\Id}{Id}
\DeclareMathOperator{\Zer}{Zer}
\DeclareMathOperator{\Fix}{Fix}

\title{\LARGE \bf
Mini-batch stochastic three-operator splitting for distributed optimization
}

\author{Barbara Franci and Mathias Staudigl$^{1}$% <-this % stops a space
\thanks{*We thank Olivier Bilenne for his contribution to the early stage of this research. M. Staudigl acknowledges financial support from the FMJH Program PGMO and from the support of EDF. }% <-this % stops a space
\thanks{$^{1}$The authors are with the Department of Data Science and Knowledge Engineering, Maastricht University, P.O. Box 616, NL\textendash 6200 MD Maastricht, The Netherlands
        {\tt\small b.franci@maastrichtuniversity.nl, m.staudigl@maastrichtuniversity.nl}}}%}

\begin{document}

\maketitle
\thispagestyle{empty}
\pagestyle{empty}

%%%%%%%%%%%%%%%%%%%%%%%%%%%%%%%%%%%%%%%%%%%%%%%%%%%%%%%%%%%%%%%%%%%%%%%%%%%%%%%%
\begin{abstract}
We consider a network of agents, each with its own private cost consisting of a sum of two possibly nonsmooth convex functions, one of which is composed with a linear operator. At every iteration each agent performs local calculations and can only communicate with its neighbors. The challenging aspect of our study is that the smooth part of the private cost function is given as an expected value and agents only have access to this part of the problem formulation via a heavy-tailed stochastic oracle. To tackle such sampling-based optimization problems, we propose a stochastic extension of the triangular pre-conditioned primal-dual algorithm. We demonstrate almost sure convergence of the scheme and validate the performance of the method via numerical experiments. 
%This paper considers general multi-agent distributed structured convex optimization problems under stochastic uncertainty. Each local objective function consists of a smooth differentiable part, given in expected-value form, a non-smooth convex term, and a linear composite function. For such problems we derive a synchronous distributed primal-dual splitting scheme which is the stochastic version of the recently introduced asymmetric forward-backward splitting method. We demonstrate almost sure convergence of the algorithm and we validate the performance of the method by numerical experiments. 
\end{abstract}

%%%%%%%%%%%%%%%%%%%%%%%%%%%%%%%%%%%%%%%%%%%%%%%%%%%%%%%%%%%%%%%%%%%%%%%%%%%%%%%%
\section{Introduction}
%----------------------------------------------------------------------
%%% Introduction
%----------------------------------------------------------------------
% !TEX root = ./STRIPD-CSS.tex

We consider a large class of convex optimization problems given by 
\begin{equation}\label{eq:Opt}
\min_{x\in\Rn}f(x)+g(x)+h(Lx),
\end{equation}
where $f:\Rn\to\R$ is a convex and continuously differentiable function, $g:\Rn\to(-\infty,\infty]$ and $h:\R^{m}\to(-\infty,\infty]$ are closed convex and lower semi-continuous functions, and $L:\Rn\to\R^{m}$ is a given linear map. Such a structure is very general, and and describes many applications that range from signal processing to machine learning to control \cite{Jin:2021wk,SLS15,li2021}. In many instances of problem \eqref{eq:Opt}, the cost functions are contaminated by stochastic noise. In such settings, we are given a probability space $(\Omega,\scrA,\Pr)$ carrying a random variable $\xi:\Omega\to\Xi\subset\R^{d}$, and a measurable function $F:\Rn\times\Xi\to\R$ so that 
\begin{equation}\label{eq:f}
f(x)=\Ex[F(x,\xi)]\qquad\forall x\in\Rn.
\end{equation}
The presence of stochastic uncertainty challenges any direct solution method for problem \eqref{eq:Opt}, since the smooth function $f$ is not directly accessible in practice, unless the distribution of the random variable is known. Indeed, the expectation \eqref{eq:f} cannot be computed exactly, and instead we need to restore to simulation-based techniques. A reasonable assumption is that we can draw samples from the distribution of the random variable $\xi$, and stochastically approximate the necessary information about the function $f$. Specifically, we adopt an online gradient-based stochastic approximation method where a deterministic version of a numerical algorithm, known to solve problem \eqref{eq:Opt} in its expected-value formulation, is supplied with stochastic estimators of the gradient of $f$. In this setting, the mechanism to access $f$ via samples of the law of $\xi$ is usually named a \emph{stochastic oracle} (SO). The SO outputs unbiased approximations of the gradient obtained via an average over a batch of realizations. The larger the number of samples, the smaller the variance of point estimators, and thus the more precise information we obtain at the cost of generating a larger amount of i.i.d random variables. This trade-off between accuracy of estimators and the available simulation budget is what makes such \emph{mini-batch} stochastic approximation approaches efficient methods of choice \cite{IusJofOliTho17,IusJofOliTho19,Bot:2021vf}.

\subsection{Main Contributions and relation to the literature}
A standard assumption in stochastic optimization is that the approximation error is uniformly bounded \cite{IusJofOliTho19}. 
%This is satisfied when noise enters the problem in form of a uniform additive noise, but not if noise is multiplicative. An in-depth discussion of this point is given in \cite{IusJofOliTho19}. 
Instead, the stochastic approximation approach developed in this paper allows to handle stochastic oracles with potentially unbounded moments. This is of relevance in primal-dual dynamics, which usually act on unbounded domains, for which any a-priori variance bound is rather unnatural. 

Besides weaker hypothesis on the noise structure, this paper is the first stochastic primal-dual algorithm which is even provably convergent in the multi-agent formulation of problem \eqref{eq:Opt}, i.e., when a finite set of agents $i\in\{1,\ldots,m\}$ cooperatively minimize the composite objective function 
\begin{equation}\label{eq:DistOpt}
\begin{split}
\min_{x_{i}\in\R^{p_{i}},1\leq i\leq m}\sum_{i=1}^{m}f_{i}(x_{i})+g_{i}(x_{i})+h_{i}(L_{i}x_{i}),\\
\text{s.t.: }A_{ij}x_{i}+A_{ji}x_{j}=b_{(i,j)}\quad\forall (i,j)\in E
\end{split}
\end{equation}
Problems of this form appear in several application fields. In distributed model predictive control, $f_{i}$ can represent individual finite-horizon costs for each agent, $L_{i}$ model linear dynamics of each agent, and $h_{i}$ models state and input constraints \cite{Jin:2021wk}. In machine learning, $f_{i}$ would represent a smooth data fidelity terms, $Lx$ linear restrictions on the parameters and $h_{i}$ and $g_{i}$ take over the role of statistical penalties reflecting a-priori structure properties of the parameters to be estimated \cite{SLS15}.

Compared with existing work, this paper makes the following contributions:
\begin{enumerate}
\item[(i)] Our algorithm is a stochastic extension of the triangular preconditioned primal-dual algorithm (TriPD), developed in \cite{TriPD,LatPat17}. 
\item[(ii)] To the best of our knowledge, our scheme is the first stochastic approximation method which is able to solve the three-operator splitting problem characterizing primal-dual pairs for \eqref{eq:Opt}.
\item[(iii)] The analysis immediately generalizes to the multi-agent formulation \eqref{eq:DistOpt}, where distributed iterations, i.e., locally performed without central supervision, are obtained. 
\end{enumerate}

The only related contributions we are aware of are \cite{YurBanCevNIPS16,ZhaoCev18}. Both assume a uniformly bounded SO, which is very restrictive in primal-dual methods, essentially forcing a-priori compactness on the domain of $g$. For the special case when $g=0$, stochastic accelerated algorithms for the centralized problem \eqref{eq:Opt} have been considered in \cite{ChenLanOu14}, imposing a uniformly bounded noise condition on the SO. We include the non-smooth term $g$ and allow for heavy-tailed noise in the SO.
%\item Our convergence proof applies also to the distributed optimization problem in \eqref{eq:DistOpt}. We are not aware of a study applicable to such a general class of multi-agent optimization problems under such weak assumptions on the noise.

Our analysis is restricted to synchronous versions of distributed optimization algorithms. It is possible to extend our analysis to the asynchronous case and block-coordinate descent strategies, acting on general real separable Hilbert spaces. We will present these extensions in a future paper. %The paper is organized as follows. Section \ref{sec:central} is devoted to the solution of the central optimization problem \eqref{eq:Opt} by the proposed method.

\subsection{Notation and preliminary results}
Throughout $\scrH,\scrG$ are finite dimensional Euclidean spaces. Their scalar products and associated norms are denoted by $\inner{\cdot,\cdot}$ and $\norm{\cdot}$. We denote $\scrB(\scrH,\scrG)$ the space of bounded linear operators from $\scrH$ to $\scrG$. The adjoint of $L\in\scrB(\scrH,\scrG)$ is denoted by $L^{\ast}$.  $\Id$ denotes the identity operator. We set $\scrP_{\alpha}(\scrH)\eqdef\{\bQ\in\scrB(\scrH)\vert \bQ^{\ast}=\bQ\text{ and }\bQ\succ\alpha\Id\}.$ For $\alpha>0$ and $\bQ\in\scrP_{\alpha}(\scrH)$, we define a scalar product and norm on $\scrH$ by 
$\inner{x,y}_{\bQ}\eqdef\inner{\bQ x,y}$, and $\norm{x}_{\bQ}\eqdef\sqrt{\inner{\bQ x,x}}$ for all $x,y\in\scrH$. We let $\cM(\bQ)=\max_{u:\norm{u}=1}\inner{u,\bQ u}$ and $\cm(\bQ)=\min_{u:\norm{u}=1} \inner{u,\bQ u}$. For an extended-valued real function $f$, we use $\dom f=\{x\in\scrH\vert f(x)<\infty\}$ for its effective domain. Let $\bQ\in\scrP_{\alpha}(\scrH)$ for some $\alpha>0$. $\delta_{C}$ is the indicator function of the set $C$, that is, $\delta_{C}(x)=0$ if $x\in C$ and $\delta_{C}(x)=\infty$ otherwise. The weighted prox-operator of $f$ is defined as $
\prox^{\bQ}_{f}(x)\eqdef \argmin_{u\in\scrH}\left\{f(u)+\frac{1}{2}\norm{u-x}^{2}_{\bQ}\right\}.$
%&=(\Id+\bQ^{-1}\partial f)^{-1}(x).
%\end{align}
The conjugate of $f:\scrH\to(-\infty,\infty]$ is $f^{\ast}(u)=\sup_{x\in\scrH}\{\inner{x,u}-f(x)\}.$ We also need the celebrated Robbins-Siegmund Lemma for the convergence analysis 
\begin{lemma}[\cite{RS71}]
\label{lem:RS}
Let $(\Omega,\scrA,(\scrA_{n})_{n\in\N},\Pr)$ be a filtered probability space satisfying the usual conditions. For every $n\in\N$, let $v_{n},\xi_{n},\zeta_{n}$ and $t_{n}$ be non-negative $\scrA_{n}$-measurable random variables such that $\{\zeta_{n}\}_{n\in\N}$ and $\{t_{n}\}_{n\in\N}$ are summable and for all $n\in\N$,
\begin{equation}
\Ex[v_{n}\vert\scrF_{n}]\leq(1+t_{n})v_{n}+\zeta_{n}-\xi_{n}\qquad\Pr-\text{a.s.}
\end{equation}
Then $\{v_{n}\}_{n\in\N}$ converges and $\{\xi_{n}\}_{n\in\N}$ is summable $\Pr$-a.s. 
\end{lemma}

\section{Stochastic Primal-Dual Algorithm}
\label{sec:algorithm}
%----------------------------------------------------------------------
%%% Algorithm
%----------------------------------------------------------------------
% !TEX root = ./STRIPD-CSS.tex

In this section we propose a stochastic primal-dual algorithm for solving \eqref{eq:Opt}. Let $\scrH\equiv\Rn$ and $\scrG\equiv\R^{m}$ with inner products $\inner{\cdot,\cdot}_{\scrH},\inner{\cdot,\cdot}_{\scrG}$, and corresponding norms $\norm{\cdot}_{\scrH},\norm{\cdot}_{\scrG}$.

\begin{assumption}\label{ass:1}
Throughout the paper the following assumptions shall be in place:
\begin{itemize}
\item[(i)] $g:\scrH\to(-\infty,\infty],h:\scrG\to(-\infty,\infty]$ are proper, closed, convex and lower-semi continuous functions.
\item[(ii)] $L:\scrH\to\scrG$ is a linear mapping with adjoint $L^{\ast}$. %and norm $\norm{L}_{\scrB(\scrH,\scrG)}\equiv\norm{L}\eqdef \sup_{x\neq 0}\tfrac{\norm{Lx}_{\scrG}}{\norm{x}_{\scrH}}$.
\item[(iii)] $f:\scrH\to\R$ is convex, continuously differentiable and there exists $\alpha>0,\bQ\in\scrP_{\alpha}(\scrH)$, and $\beta_{f}>0$ such that $
\norm{\nabla f(x)-\nabla f(y)}_{\bQ^{-1}}\leq \beta_{f}\norm{x-y}_{\bQ}$ for all $x,y\in\scrH.$
\item[(iv)] Let $\Xi\subset\R^{d}$ be a measurable set and $(\Omega,\scrA,\Pr)$ a probability space. There exists a measurable function $F:\scrH\times\Xi\to\R$ such that \eqref{eq:f} holds.
\item[(v)] The set of solutions to \eqref{eq:Opt}, denoted by $\scrX^{\ast}$, is nonempty Moreover, there exists $x\in\rint(\dom g)$ such that $Lx\in\rint(\dom h)$.
\end{itemize}
\end{assumption}
Let $\scrZ=\scrG\times\scrH$ represent the product space with the inner product 
$\inner{(y_{1},x_{1}),(y_{2},x_{2})}=\inner{y_{1},y_{2}}_{\scrG}+\inner{x_{1},x_{2}}_{\scrH},$ 
and associated norm $\norm{(y,x)}=\sqrt{\norm{y}^{2}_{\scrG}+\norm{x}^{2}_{\scrH}}.$ Whenever clear from the context, we will suppress the ambient space from the inner products and norms, respectively.
 
We remark that by the Baillon-Haddad theorem \cite[Corollary 18.17]{BauCom16} the $\beta_{f}$-Lipschitz smoothness of the function $f$ is equivalent to the $\tfrac{1}{\beta_{f}}$-cocoercivity of $\nabla f$, i.e.,
\begin{equation*}\label{eq:Baillon}
\inner{\nabla f(x)-\nabla f(x'),x-x'}\geq \tfrac{1}{\beta_{f}}\norm{\nabla f(x)-\nabla f(x')}^{2}.
\end{equation*}
for all $x,x'\in\scrH$.

Problem \eqref{eq:Opt} can be equivalently written as a max-min problem 
\begin{equation*}
\max_{y\in\scrG}\min_{x\in\scrH}\{f(x)+g(x)+\inner{Lx,y}-h^{\ast}(y)\}
\end{equation*}
Then, $(\bar{y},\bar{x})\in\scrG\times\scrH$ is called a primal-dual pair if
\begin{equation*}
\begin{split}
0&\in \partial h^{\ast}(\bar{y})-L\bar{x},\\
0&\in\partial g(\bar{x})+\nabla f(\bar{x})+L^{\ast}\bar{y}.
\end{split}
\end{equation*} 

Introduce the maximally monotone operators 
$$
\begin{aligned}
&\setA(y,x)\eqdef\partial h^{\ast}(y)\times\partial g(x),\\
&\setM(y,x)\eqdef[-Lx,L^{\ast}y],\\
&\setC(y,x)\eqdef[0,\nabla f(x)].
\end{aligned}
$$
Then $\bar{z}\eqdef (\bar{y},\bar{x})$ is a primal-dual pair if and only if 
$\bar{z}\in\Zer(\setG)\eqdef\{z\in\scrZ\vert 0\in\setG(z)\},$ where $\setG(y,x)\eqdef\setA(y,x)+\setM(y,x)+\setC(y,x).$
%The main aim of this paper is to provide a stochastic analysis of the Triangularly preconditioned primal-dual algorithm (TriPD) \cite{TriPD} in case where the single-valued operator $\scrC$ is only available after consulting a stochastic oracle. Before that, we recall the deterministic algorithm, due to \cite{TriPD}.\\

\subsection{Triangular pre-conditioned primal-dual algorithm}
%In the current formulation, the major obstacle of computing a primal-dual pair is that in practice we cannot directly access the operator $\setC$. 
If it is possible to access the operator $\setC$ directly, then an application of the Triangular pre-conditioned Primal-Dual (TriPD) algorithm of \cite{TriPD} would be possible. 
Let $\tau_{1}>0$ and $\tau_{2}>0$ be two positive real numbers, and $\Sigma\in\scrP_{1/\tau_{1}}(\scrG),\Gamma\in\scrP_{1/\tau_{2}}(\scrH)$.  
%
%We are given a sequence of bounded linear self-adjoint operators $\Sigma_{k}\in\scrP_{1/\tau_{1}}(\scrG)$ and $\Gamma_{k}\in\scrP_{1/\tau_{2}}(\scrH)$, which act as variable metric applied to the updates in the space $\scrG_{k}\eqdef (\scrG,\norm{\cdot}_{\Sigma_{k}^{-1}})$ and $\scrH_{k}\eqdef(\scrH,\norm{\cdot}_{\Gamma_{k}^{-1}})$, respectively. 
%%%%%%%%%%%%%%%%%%%%%%%%
%\begin{algorithm}[t]
%\caption{Triangularly preconditioned primal-dual algorithm (TriPD)}
%\label{alg:TriPD}
%\SetAlgoLined
%\KwData{$x_{0}\in\scrH,y_{0}\in\scrG,(\Sigma_{k})_{k\geq 0}\subset\scrP_{1/\tau_{1}}(\scrG),(\Gamma_{k})_{k\geq 0}\in\scrP_{1/\tau_{2}}(\scrH)$}
%\While{$k=0,\dots,$}
%{
%\begin{align*}
%\bar{y}_{k}&=\prox^{\Sigma^{-1}_{k}}_{h^{\ast}}(y_{k}+\Sigma_{k}Lx_{k})\\
%x_{k+1}&=\prox_{g}^{\Gamma^{-1}_{k}}(x_{k}-\Gamma_{k}\nabla f(x_{k})-\Gamma_{k}L^{\ast}\bar{y}_{k})\\
%y_{k+1}&=\bar{y}_{k}+\Sigma_{k}L(x_{k+1}-x_{k})
%\end{align*}
%}
%\end{algorithm}
%%%%%%%%%%%%%%%%%%%%%%%%%
TriPD can be compactly written as the fixed point iteration 
$z_{k+1}=T(z_{k})$, where $z_{k}\eqdef(y_{k},x_{k})\in\scrZ$, and the mapping $T:\scrZ\to\scrZ$ is given by 
\begin{equation}\label{eq_tripd}
\begin{aligned}
&\Phi_{1}(z)\eqdef\prox^{\Sigma^{-1}}_{h^{\ast}}(y+\Sigma Lx),\\
&\Phi_{2}(z)\eqdef\prox^{\Gamma^{-1}}_{g}(x-\Gamma\nabla f(x)-\Gamma\Phi_{1}(z))\\
&T(z)\eqdef[\Phi_{1}(z)+\Sigma L(\Phi_{2}(z)-x),\Phi_{2}(z)].
\end{aligned}
\end{equation}
%We emphasize that this definition dictates a sequential implementation: First, we need to evaluate the map $\Phi_{1}:\scrZ\to\scrG$ at the current position $z\in\scrZ$. Then we need to compute $\Phi_{2}:\scrZ\to\scrH$, which requires $\Phi_{1}(z)$ implicitly. Finally, we combine the two actions to generate the next iterate $z^{+}=T(z)$.
First, let us report an important connection between the fixed points of the mapping $T$ and the set of primal-dual solutions to \eqref{eq:Opt}.
\begin{lemma}\cite[Equation (15)]{TriPD}.
We have  $\Fix(T)\eqdef\{z\in\scrZ\vert T(z)=z\}=\Zer(\setG)$.
\end{lemma}

Before taking care of our expected valued formulation as in \eqref{eq:f}, let us introduce the matrices 
\begin{align*}
&\bP\eqdef\left(\begin{array}{cc} 
\Sigma^{-1} & \tfrac{1}{2}L\\
\tfrac{1}{2}L^{\ast} & \Gamma^{-1}\end{array}\right), \bK\eqdef\left(\begin{array}{cc} 0 & -\tfrac{1}{2}L \\ \tfrac{1}{2}L^{\ast} & 0\end{array}\right),\\
&\bS\eqdef\left(\begin{array}{cc} \Sigma^{-1} & 0 \\ 0 & \Gamma^{-1}\end{array}\right), \bR\eqdef\bP+\bK.
\end{align*}
These matrices act like step-sizes and pre-conditioners. Allowing for matrix-valued step sizes makes the distributed algorithm in Section \ref{sec:distributed} a corollary of the present analysis.
%
%Using all the definitions above, the operator $T$ can be be written as 
%\begin{equation}\label{eq:Tdifference}
%T(z)=z+\bS^{-1}(\bR+\setM^{\ast})(\Phi(z)-z),
%\end{equation}
%where $\bR\eqdef \bP+\bK$, and
%$$
%\Phi(z)=(\bR+\setA)^{-1}(\bR-\setM-\setC)(z).
%$$
%We report an important connection between the fixed points of the mapping $T$ and the set of primal-dual solutions to \eqref{eq:Opt}.
%\begin{lemma}[\cite{TriPD}]
%We have  $\Fix(T)\eqdef\{z\in\scrZ\vert T(z)=z\}=\Zer(\setG)$.
%\end{lemma}

%%%%%%
\begin{algorithm}[t]
\caption{Stochastic Triangularly preconditioned primal-dual algorithm (STriPD)}
\label{alg:STriPD}
\SetAlgoLined
Initialize: $(x_{0},y_{0})\in\scrZ$, $\Sigma\in\scrP_{1/\tau_{1}}(\scrG)$, $\Gamma\in\scrP_{1/\tau_{2}}(\scrH)$\\
Iteration $k$:
\begin{align*}
\hat{Y}_{k}&=\prox^{\Sigma^{-1}}_{h^{\ast}}(Y_{k}+\Sigma LX_{k})\\
X_{k+1}&=\prox_{g}^{\Gamma^{-1}}(x_{k}-\Gamma\scrF_{k}(X_{k})-\Gamma L^{\ast}\hat{Y}_{k})\\
Y_{k+1}&=\hat{Y}_{k}+\Sigma L(X_{k+1}-X_{k})
\end{align*}
\end{algorithm}
%%%%%%%%%%%%%%%%%%%%%%%%%

\subsection{Stochastic TriPD}
Since we cannot evaluate $\setC(z)$ directly, we let $\scrC_{k}(z)$ denote its stochastic estimator. This estimator is constructed by a Monte-Carlo scheme involving mini-batches. Given an i.i.d. sample $\xi_{1:N}\eqdef \{\xi^{i}\}_{i=1}^{N}$ drawn from the law of $\xi$, let 
$$
\bar{w}(x,\xi_{1:N})\eqdef \tfrac{1}{N}\textstyle{\sum_{i=1}^{N}}(\nabla_{x}F(x,\xi^{i})-\nabla f(x)).
$$
The approximation take place at each iteration $k$ and a sequence $\{N_{k}\}_{k\in\N}$ (the \emph{batch size}) defines the number of random variables we need to sample in each iteration.
\begin{assumption}\label{ass:batch}
We are given a sequence $\{N_{k}\}_{k\in\N}\subset\N$ such that $\lim_{k\to\infty}N_{k}=\infty$. 
\end{assumption}
Since complexity questions are beyond our scope, we do not specify the speed at which batch sizes grow (see, however, \cite{Bot:2021vf}). For a given sequence of batch-sizes, we set $w_{k}(x,\omega)\eqdef\bar{w}(x,\xi_{1:N_{k}}(\omega))$.  
\begin{assumption}\label{ass:C}
For all $k\geq 0$, $\scrC_{k}(z,\omega)=[0,\scrF_{k}(x,\omega)]$, where 
$\scrF_{k}(x,\omega)=\nabla f(x)+w_{k}(x,\omega)$ for all $(x,\omega)\in\scrH\times\Omega.$
\end{assumption}
We let $(\scrA_{k})_{k\geq 0}$ denote the filtration given by $\scrA_{0}\eqdef \{\emptyset,\Omega\}$, and for $k\geq 1$, $\scrA_{k}\eqdef\sigma(\xi_{1:N_{0}},\ldots,\xi_{1:N_{k-1}})$. Define $\nu(x)\eqdef \norm{\nabla_{x}F(x,\xi)-\nabla f(x)}^{2}$, and 
$\sigma(x)\eqdef \sqrt{\Ex[\nu(x)\vert x]}.$
\begin{assumption}\label{ass:noise}
For all $k\geq 0$ we have $\Ex[\scrF_{k}(X_{k})\vert\scrA_{k}]=\nabla f(X_{k})$ a.s.. Moreover, there exists $x^{\ast}\in\scrX^{\ast}$ and $\sigma_{0}(x^{\ast})\geq 0,\sigma_{1}>0$ such that 
\begin{equation*}
\sigma(x)\leq\sigma_{0}(x^{\ast})+\sigma_{1}\norm{x-x^{\ast}}^{2}\quad\forall x\in\scrH.
\end{equation*}
\end{assumption}
This is a heavy-tailed noise assumption which allows us to consider random perturbation even with unbounded second moment \cite{IusJofOliTho17,JofTho19}.  %For instance, this weakened hypothesis gains relevance in PDE-constrained optimization problems under uncertainty, as well as in statistical estimation problems involving reproducing kernel Hilbert spaces \cite{DieBac16}. 

In analogy to TriPD and in light of the approximation scheme, the equations in \eqref{eq_tripd} correspond to those in Algorithm \ref{alg:STriPD} and lead to a sequence of random maps $\{\T_{k}\}_{k\geq0}:\scrZ\times\Omega\to\scrZ$ generating a stochastic process $\{Z_{k}\}_{k\geq 0}=\{(Y_{k},X_{k});k\geq 0\}$ via recursive updates 
\begin{equation*}\label{eq:Zrandom}
Z_{k+1}=\T_{k}(Z_{k}),\quad Z_{0}\in\scrZ\text{ given.}
\end{equation*}
The exact action of this mapping can be described as follows. For $z=(y,x)\in\scrZ$, let
\begin{align*}
&\Y_{k}(z)\eqdef\prox^{\Sigma^{-1}}_{h^{\ast}}(y+\Sigma Lx),\\
&\X_{k}(z)\eqdef\prox^{\Gamma^{-1}}_{g}\left(x-\Gamma\scrF_{k}(x)-\Gamma\Y_{k}(z)\right)\\
&\T_{k}(z)\eqdef\left[\Y_{k}(z)+\Sigma L\left(\X_{k}(z)-x\right),\X_{k}(z)\right].
\end{align*}

Let $\hat{Z}_{k}\eqdef[\Y_{k}(z),\X_{k}(z)]$, so that $Y_{k+1}=\Y_{k}(Z_{k})+\Sigma_{k}L\left(\X_{k}(z)-X_{k}\right)$  and $X_{k+1}=\X_{k}(Z_{k})$ for all $k\geq 0$. One can verify that 
\begin{align*}%\label{eq:hatZ}
&\hat{Z}_{k}=(\bR+\setA)^{-1}(\bR-\setM-\scrC_{k})(Z_{k}), \text{ and} \\
&\T_{k}(z)=z+\bS^{-1}(\bR+\setM^{\ast})(\hat{Z}_{k}-z).
%\label{eq:randomTdifference}
\end{align*}
%%%%%%%%%%%%%%%%%%

\section{Convergence analysis}
\label{sec:analysis}
%----------------------------------------------------------------------
%%% Analysis
%----------------------------------------------------------------------
% !TEX root = ./STRIPD-CSS.tex

%In this section we study the convergence of the stochastic process $(Z_{k})_{k\geq 0}$ generated by the recursion \eqref{eq:Zrandom}
In this section, we present a number of results that lead to the convergence proof of Algorithm \ref{alg:STriPD} (Theorem \ref{thm:main}). We start with a property of the operator $\setC$.
\begin{lemma}\cite[Lemma II.4]{TriPD}.
\label{lem:C}
For all $z=(y,x)$, $z'=(y',x')$, $z''=(y'',x'')\in\scrZ$ we have 
\begin{equation*}
\inner{\setC(z)-\setC(z'),z''-z}\leq \tfrac{\beta_{f}}{4}\norm{x''-x'}^{2}_{\bQ}. 
\end{equation*}
\end{lemma}
%\begin{proof}
%See Lemma II.4 of \cite{TriPD}.
%\end{proof}
%%%%%%%%%%%%
From Assumption \ref{ass:C}, we can write 
\[
\scrC_{k}(z)=\setC(z)+\phi_{k+1}(z),
\]
where $\phi_{k+1}(z,\omega)\eqdef [0,w_{k}(x,\omega)]$. Let $z^{\ast}\in\Zer(\setG)$, or $-(\setM+\setC)(z^{\ast})\in \setA(z^{\ast})$. From the definition of the update $\hat{Z}_{k}$, we deduce $\bR(Z_{k}-\hat{Z}_{k})-(\setM+\scrC_{k})(Z_{k})\in \setA(\hat{Z}_{k})$. By monotonicity of the operator $\setA$, this implies 
\begin{equation}\label{eq_A}
\begin{aligned}
&\inner{\bR(Z_{k}-\hat{Z}_{k})-\setM(Z_{k}-z^{\ast})+\setC(z^{\ast})-\setC(Z_{k}),\hat{Z}_{k}-z^{\ast}}\\
&-\inner{\phi_{k+1}(Z_{k}),\hat{Z}_{k}-z^{\ast}}\geq 0.
\end{aligned}
\end{equation}
%\begin{lemma}\label{lem:C}
%For all $z=(y,x),z'=(y',x'),z''=(y'',x'')\in\scrZ$ we have 
%\begin{equation}
%\inner{\setC(z)-\setC(z'),z''-z}\leq \tfrac{\beta_{f}}{4}\norm{x''-x'}^{2}_{\bQ}. 
%\end{equation}
%\end{lemma}
%\begin{proof}
%See the appendix.
%By Fenchel-Young inequality and the Baillon-Haddad Theorem \eqref{eq:Baillon}, we get 
%\begin{align*}
%\inner{\setC(z)-\setC(z'),z''-z}&=\inner{\nabla f(x')-\nabla f(x),x-x''}\\
%&=\inner{\nabla f(x')-\nabla f(x),x-x'}+\inner{\nabla f(x)-\nabla f(x'),x'-x''}\\
%&\leq \inner{\nabla f(x')-\nabla f(x),x-x'}+\tfrac{1}{\beta_{f}}\norm{\nabla f(x)-\nabla f(x')}^{2}_{\bQ^{-1}}+\tfrac{\beta_{f}}{4}\norm{x''-x'}^{2}_{\bQ}\\
%&\leq \inner{\nabla f(x')-\nabla f(x),x-x'}+\inner{\nabla f(x)-\nabla f(x'),x-x'}+\tfrac{\beta_{f}}{4}\norm{x''-x'}^{2}_{\bQ}\\
%&\leq \tfrac{\beta_{f}}{4}\norm{x''-x'}^{2}_{\bQ}. 
%\end{align*}
%\end{proof}
Applying Lemma \ref{lem:C} to the points $z=z^{\ast},z'=Z_{k}$ and $z''=\hat{Z}_{k}$, we get $\inner{\setC(z^{\ast})-\setC(Z_{k}),\hat{Z}_{k}-z^{\ast}}\leq \tfrac{\beta_{f}}{4}\norm{\X_{k}(Z_{k})-X_{k}}^{2}_{\bQ}.$
To reduce notational clutter, we write in the following $\hat{X}_{k}\equiv\X_{k}(Z_{k})$. Then, \eqref{eq_A} leads to the estimate 
\begin{align*}
0&\leq \inner{\bR(Z_{k}-\hat{Z}_{k})+\setM(z^{\ast}-Z_{k}),\hat{Z}_{k}-z^{\ast}}\\
&+\inner{\phi_{k+1}(Z_{k}),z^{\ast}-\hat{Z}_{k}}+\tfrac{\beta_{f}}{4}\norm{\hat{X}_{k}-X_{k}}^{2}_{\bQ}\\
%&=\inner{(\setM-\bK)(Z_{k}-z^{\ast}),z^{\ast}-\hat{Z}_{k}}+\inner{\bP(\hat{Z}_{k}-Z_{k}),z^{\ast}-\hat{Z}_{k}}\\
%&+\inner{\phi_{k+1}(Z_{k}),z^{\ast}-\hat{Z}_{k}}+\tfrac{\beta_{f}}{4}\norm{\hat{X}_{k}-X_{k}}^{2}_{\bQ}\\
%&=\inner{(\setM-\bK)(Z_{k}-z^{\ast})+\bP_{k}(\hat{Z}_{k}-Z_{k}),z^{\ast}-Z_{k}}+\inner{(\setM-\bK)(Z_{k}-z^{\ast}),Z_{k}-\hat{Z}_{k}}-\norm{\hat{Z}_{k}-Z_{k}}_{\bP_{k}}^{2}\\
%&+\inner{\phi_{k+1}(Z_{k}),z^{\ast}-\hat{Z}_{k}}+\tfrac{\beta_{f}}{4}\norm{\hat{X}_{k}-X_{k}}^{2}_{\bQ}\\
&=\inner{\bP(\hat{Z}_{k}-Z_{k}),z^{\ast}-Z_{k}}\\
&+\inner{(\setM-\bK)(Z_{k}-z^{\ast}),Z_{k}-\hat{Z}_{k}}-\norm{\hat{Z}_{k}-Z_{k}}_{\bP}^{2}\\
&+\inner{\phi_{k+1}(Z_{k}),z^{\ast}-\hat{Z}_{k}}+\tfrac{\beta_{f}}{4}\norm{\hat{X}_{k}-X_{k}}^{2}_{\bQ}
\end{align*}
where the last equality uses the skew-symmetry so that $\inner{(\setM-\bK)z,z}=0$ for all $z\in\scrZ$. Thanks to the skew-symmetry, we also observe that 
\begin{align*}
&\inner{(\setM-\bK)(Z_{k}-z^{\ast}),Z_{k}-\hat{Z}_{k}}=%\inner{Z_{k}-z^{\ast},\setM^{\ast}(Z_{k}-\hat{Z}_{k})}\\
%&-\inner{\bK(Z_{k}-z^{\ast}),Z_{k}-\hat{Z}_{k}}
%&=\inner{z^{\ast}-Z_{k},\setM^{\ast}(\hat{Z}_{k}-Z_{k})}+\inner{\bK(\hat{Z}_{k}-Z_{k}),z^{\ast}-Z_{k}}\\
\inner{z^{\ast}-Z_{k},\setM^{\ast}(\hat{Z}_{k}-Z_{k})}\\
&+\inner{(\bR-\bP)(\hat{Z}_{k}-Z_{k}),z^{\ast}-Z_{k}}.
\end{align*}
Therefore, 
\begin{align*}
0&\leq \inner{(\bR+\setM^{\ast})(\hat{Z}_{k}-Z_{k}),z^{\ast}-Z_{k}}-\norm{\hat{Z}_{k}-Z_{k}}^{2}_{\bP}\\
&+\inner{\phi_{k+1}(Z_{k}),z^{\ast}-\hat{Z}_{k}}+\tfrac{\beta_{f}}{4}\norm{\hat{X}_{k}-X_{k}}^{2}_{\bQ}.
\end{align*}
By definition $\bS(\T_{k}(Z_{k})-Z_{k})=(\bR+\setM^{\ast})(\hat{Z}_{k}-Z_{k})$, and
\begin{equation}\label{eq_B}
\begin{aligned}
0&\leq \inner{\bS(\T_{k}(Z_{k})-Z_{k}),z^{\ast}-Z_{k}}-\norm{\hat{Z}_{k}-Z_{k}}^{2}_{\bP}\\
&+\inner{\phi_{k+1}(Z_{k}),z^{\ast}-\hat{Z}_{k}}+\tfrac{\beta_{f}}{4}\norm{\hat{X}_{k}-X_{k}}^{2}_{\bQ}.
\end{aligned}
\end{equation}
A straightforward, but slightly tedious computation, yields the next result. %We leave the verification to the reader.
\begin{lemma}\label{lem:hatZ}
Let $\bU\eqdef \left(\begin{array}{cc} \Sigma^{-1} & -\tfrac{1}{2}L \\ -\tfrac{1}{2}L^{\ast} & \Gamma^{-1}-\tfrac{\beta_{f}}{4}\bQ\end{array}\right)$. Then, 
$$
\norm{\hat{Z}_{k}-Z_{k}}^{2}_{\bP_{k}}-\tfrac{\beta_{f}}{4}\norm{\hat{X}_{k}-X_{k}}^{2}_{\bQ}=\norm{\T_{k}(Z_{k})-Z_{k}}^{2}_{\bU}.
$$
\end{lemma}
Eq. \eqref{eq_B} and Lemma \ref{lem:hatZ} deliver the relation 
\begin{equation}\label{eq_C}
\begin{aligned}
0&\leq\inner{\bS(\T_{k}(Z_{k})-Z_{k}),z^{\ast}-Z_{k}}\\
&+\inner{\phi_{k+1}(Z_{k}),z^{\ast}-\hat{Z}_{k}}-\norm{\T_{k}(Z_{k})-Z_{k}}^{2}_{\bU}. 
\end{aligned}
\end{equation}
We now analyze the noise term $\inner{\phi_{k+1}(Z_{k}),z^{\ast}-\hat{Z}_{k}}$. 
Denote by $\bar{x}_{k}\eqdef\Phi_{2}(Z_{k})$ the evaluation of the deterministic generator at $Z_{k}$ as in \eqref{eq_tripd}. Then, 
\begin{align*}
\inner{\phi_{k+1}(Z_{k}),z^{\ast}-\hat{Z}_{k}}&=\inner{\scrF(X_{k})-\nabla f(X_{k}),x^{\ast}-\hat{X}_{k}}\\
&=\inner{\scrF_{k}(X_{k})-\nabla f(X_{k}),x^{\ast}-\bar{x}_{k}}\\
&-\inner{\scrF_{k}(X_{k})-\nabla f(X_{k}),\hat{X}_{k}-\bar{x}_{k}}.
\end{align*}
Using Cauchy-Schwarz, and the non-expansiveness of the proximal operator $\prox^{\Gamma^{-1}}_{g}(\cdot)$ with respect to the norm $\norm{\cdot}_{\Gamma}$ \cite{BauCom16}, we obtain
\begin{align*}
&-\inner{\scrF_{k}(X_{k})-\nabla f(X_{k}),\hat{X}_{k}-\bar{x}_{k}}\\%\norm{\scrF_{k}(X_{k})-\nabla f(X_{k})}_{\Gamma_{k}^{-1}}\cdot\norm{\hat{X}_{k}-\bar{x}_{k}}_{\Gamma_{k}}\\
&\leq \norm{\scrF_{k}(X_{k})-\nabla f(X_{k})}_{\Gamma}\cdot \norm{\Gamma(\scrF(X_{k})-\nabla f(X_{k}))}_{\Gamma^{-1}}\\
&=\norm{\scrF_{k}(X_{k})-\nabla f(X_{k})}^{2}_{\Gamma}.
\end{align*}
Hence 
\begin{align*}
\inner{\phi_{k+1}(Z_{k}),z^{\ast}-\hat{Z}_{k}}&\leq \inner{\scrF_{k}(X_{k})-\nabla f(X_{k}),x^{\ast}-\bar{x}_{k}}\\
&+\norm{\scrF_{k}(X_{k})-\nabla f(X_{k})}^{2}_{\Gamma}.
\end{align*}
Let $\Ex_{k}[\cdot]\eqdef \Ex[\cdot\vert\scrA_{k}]$ and recall that Assumption \ref{ass:noise} implies that $\Ex_{k}[\scrF_{k}(X_{k})]=\nabla f(X_{k})$ holds a.s.. Then, we obtain 
$$
\Ex_{k}[\inner{\phi_{k+1},z^{\ast}-\hat{Z}_{k}}]\leq \Ex_{k}[\norm{\scrF_{k}(X_{k})-\nabla f(X_{k})}^{2}_{\Gamma}].
$$
Setting $\bar{\nu}_{k}\eqdef \norm{\scrF_{k}(X_{k})-\nabla f(X_{k})}^{2}_{\Gamma}$, we obtain from \eqref{eq_C}
\begin{equation}\label{eq:estimate1}
\begin{split}
\Ex_{k}\left[\inner{\bS(\T_{k}(Z_{k})-Z_{k}),Z_{k}-z^{\ast}}\right]\leq \Ex_{k}[\bar{\nu}_{k}]\\
-\Ex_{k}\left[\norm{\T_{k}(Z_{k})-Z_{k}}^{2}_{\bU}\right].
\end{split}
\end{equation}
Since, by definition,
\begin{align*}
&\norm{Z_{k+1}-z^{\ast}}^{2}_{\bS}=\norm{(\T_{k}(Z_{k})-Z_{k})+(Z_{k}-z^{\ast})}^{2}_{\bS}\\
&=\norm{\T_{k}(Z_{k})-Z_{k}}^{2}_{\bS}+2\inner{\bS(\T_{k}(Z_{k})-Z_{k}),Z_{k}-z^{\ast}}\\
&+\norm{Z_{k}-z^{\ast}}^{2}_{\bS},
\end{align*}
the estimate \eqref{eq:estimate1} delivers 
\begin{equation*}\label{eq:recursion1}
\begin{aligned}
\Ex_{k}&\left[\norm{Z_{k+1}-z^{\ast}}^{2}_{\bS}\right]\leq \norm{Z_{k}-z^{\ast}}^{2}_{\bS}\\
&-\Ex_{k}\left[\norm{\T_{k}(Z_{k})-Z_{k}}^{2}_{2\bU_{k}-\bS}\right]+2\Ex_{k}[\bar{\nu}_{k}]
\end{aligned}
\end{equation*}
If $2\bU-\bS\succ 0$  and $\{\Ex_{k}[\nu_{k}]\}_{k\geq 0}$ is summable, then it follows that $\{Z_{k}\}_{k\geq 0}$ is quasi-Fej\'{e}r monotone with respect to $\Zer(\setG)$ relative to $\norm{\cdot}_{\bS}$, i.e.,
$$
\Ex_{k}\left[\norm{Z_{k+1}-z^{\ast}}^{2}_{\bS}\right]\leq\norm{Z_{k}-z^{\ast}}^{2}_{\bS}+2\Ex_{k}[\bar{\nu}_{k}].
$$
The next two results guarantee this.

\begin{lemma}\label{lem:variance}
Suppose that Assumption \ref{ass:noise} holds true. For all $k\in\N$, we have 
$$
\Ex_{k}[\norm{\scrF_{k}(X_{k})-\nabla f(X_{k})}^{2}_{\Gamma}]\leq a_{k}+b_{k}\norm{Z_{k}-z^{\ast}}^{2}_{\bS}\qquad\text{a.s.}
$$
where $a_{k}\eqdef 2\sigma^{2}_{0}(x^{\ast})^{2}\frac{\cM(\Gamma)}{N_{k}}$ and $b_{k}\eqdef 2\sigma_{1}^{2}\frac{\cM(\Gamma)}{\cm(\Gamma)N_{k}}$.
\end{lemma}
\begin{proof}
We compute 
\begin{align*}
&\Ex_{k}[\norm{\scrF_{k}(X_{k})-\nabla f(X_{k})}^{2}_{\Gamma}]\leq\cM(\Gamma)\Ex_{k}[\norm{\scrF_{k}(X_{k})-\nabla f(X_{k})}^{2}]\\
&=\cM(\Gamma)\Ex[\norm{w_{k}(X_{k})}^{2}\vert X_{k}]=\cM(\Gamma)\tfrac{\sigma^{2}(X_{k})}{N_{k}}\quad\text{a.s.}
\end{align*}
Using Assumption \ref{ass:noise}, we get 
\begin{align*}
\sigma^{2}(X_{k})&\leq 2\sigma^{2}_{0}(x^{\ast})+2\sigma_{1}^{2}\norm{X_{k}-x^{\ast}}^{2}\\
&\leq 2\sigma^{2}_{0}(x^{\ast})+2\sigma_{1}^{2}\tfrac{1}{\cm(\Gamma)}\norm{X_{k}-x^{\ast}}^{2}_{\Gamma}\\
&\leq 2\sigma^{2}_{0}(x^{\ast})+2\sigma_{1}^{2}\tfrac{1}{\cm(\Gamma)}\norm{Z_{k}-z^{\ast}}^{2}_{\bS} 
\end{align*}
\end{proof}
%%%%%%%%%%%%%%%%%%
\begin{lemma}\cite[Lemma 5.1]{LatPat17}
Assume that 
\begin{equation}\label{eq:step}
\cm(\Gamma^{-1}-\tfrac{\beta_{f}}{2}\bQ)>\tfrac{\norm{L}^{2}}{\cm(\Sigma^{-1})}=\norm{L}^{2}\cM(\Sigma)
\end{equation}
Then, $2\bU-\bS\in\scrP_{\tau}(\scrZ)$ for some $\tau>0$.
% with 
%\tau&=\tfrac{1}{2}\cm\left(\Gamma^{-1}-\tfrac{\beta_{f}}{2}\bQ\right)+\tfrac{1}{2}\cm(\Sigma^{-1})\\
%&-\tfrac{1}{2}\sqrt{4\norm{L}^{2}+\left(\cm(\Gamma^{-1}-\tfrac{\beta_{f}}{2}\bQ)-\cm(\Sigma^{-1})\right)^{2}}.
%\end{align*}
\end{lemma}
% \begin{proof}
%See Lemma 5.1 in \cite{LatPat17}.
% \end{proof}
 %
 %\begin{remark}
 %It can be checked that a sufficient condition for \eqref{eq:step} to hold is that 
 %\[
 %\Gamma_{k}^{-1}-\tfrac{\beta_{f}}{2}\bQ-L^{\ast}\Sigma_{k}L\succ 0,
 %\]
 %which is Assumption 2 in \cite{TriPD}.
 %\end{remark}
%
%%%

\begin{algorithm*}[t]
\caption{Distributed STriPD}
\label{alg:SynchSTriPD}
\SetAlgoLined
Initialize: $x^{0}_{i}\in\scrH_{i},y^{0}_{i}\in\scrG_{i}$ for $i=1,\ldots,m$ and $w_{(i,j)}\in\scrW^{\ast}_{(i,j)}$ for all $(i,j)\in E$.\\
Iteration $k$: For all agents $i=1,2,\ldots,m$
\begin{align*}
&\bar{w}_{(i,j),k}^{i}=\tfrac{1}{2}(w^{i}_{(i,j),k}+w^{j}_{(i,j),k})+\tfrac{\tau_{(i,j)}}{2}(A_{ij}x^{i}_{k}+A_{ji}x^{j}_{k}-b_{(i,j)})\quad\forall j\in\scrN_{i},\\
&\bar{v}^{i}_{k}=\prox_{\sigma_{i}h^{\ast}_{i}}(v_{k}^{i}+\sigma_{i}L_{i}x^{i}_{k}),\\
&x^{i}_{k+1}=\prox_{\gamma_{i}g^{\ast}_{i}}[x^{k}_{i}-\gamma \scrF^{i}_{k}(x^{k})-\gamma_{i}L^{\ast}_{i}\bar{v}_{i}^{k}-\gamma_{i}\textstyle{\sum_{j\in\scrN_{i}}}A_{ij}w^{i}_{(i,j),k}]\\
&v^{i}_{k+1}=\bar{v}^{i}_{k}+\tau_{(i,j)}(x^{i}_{k+1}-x^{i}_{k})\\
&w^{i}_{(i,j),k+1}=\bar{w}^{i}_{(i,j),k}+\tau_{i,j}A_{i,j}(x^{i}_{k+1}-x^{i}_{k})\quad\forall j\in\scrN_{i}.
\end{align*}
\end{algorithm*}

We can finally prove the main result of this paper. 
\begin{theorem}\label{thm:main}
Let Assumptions \ref{ass:1}-\ref{ass:noise} hold true. Choose $\Gamma,\Sigma$ such that condition \eqref{eq:step} holds. Then, the stochastic process $\{Z_{k}\}_{k\in\N}$ generated by Algorithm \ref{alg:STriPD} converges a.s. to a random variable $Z_{\infty}\in\Zer(\setG)$, i.e., to a solution of \eqref{eq:Opt}.
\end{theorem}
\begin{proof}
Using Lemma \ref{lem:variance}, we get 
\begin{align*}
&\Ex_{k}\left[\norm{Z_{k+1}-z^{\ast}}^{2}_{\bS}\right]\leq\norm{Z_{k}-z^{\ast}}^{2}_{\bS}+2\Ex_{k}[\bar{\nu}_{k}]\\
&-\Ex_{k}\left[\norm{\T_{k}(Z_{k})-Z_{k}}^{2}_{2\bU-\bS}\right]\\
&\leq(1+2b_{k})\norm{Z_{k}-z^{\ast}}^{2}_{\bS}-\Ex_{k}[\norm{\T_{k}(Z_{k})-Z_{k}}^{2}_{2\bU-\bS}]+2a_{k}.
\end{align*}
Set $v_{k}=\norm{Z_{k}-z^{\ast}}_{\bS}^{2}$, $t_{k}=2b_{k}$, $\zeta_{k}=2a_{k}$ and $\theta_{k}=\Ex_{k}[\norm{\T_{k}(Z_{k})-Z_{k}}^{2}_{2\bU-\bS}]$, and apply Lemma \ref{lem:RS} to deduce $\Pr\left(\sum_{k\in\N}\theta_{k}<\infty\right)=1$ and that $v_{k}$ converges to a finite random variable $v_{\infty}$ a.s. Furthermore, using \cite[Proposition 2.3]{ComPes15}, we know that $\{Z_{k}\}_{k\in\N}$ is almost surely bounded. Hence, there exists a measurable set $\Omega_{0}\subseteq\Omega$ with $\Pr(\Omega_{0})=1$ such that $\lim_{k\to\infty}\norm{Z_{k}(\omega)-z^{\ast}}=0$ for all $\omega\in\Omega_{0}$, Fix such an event $\omega\in\Omega_{0}$ and sequence $\{k_{j}\}\subset\N$ with $k_{j}\uparrow\infty$. Consider a converging subsequence $\{Z_{k_{j}}(\omega)\}_{j\in\N}$ with $Z_{k_{j}}(\omega)\to Z_\infty(\omega)$. To reduce notational clutter, we omit the relabeling and simply denote $Z_{k}(\omega)$ the converging subsequence. We have to show that $Z_\infty(\omega)\in\Zer(G)$. Set $\mu\eqdef \cm(2\bU-\bS)\geq \tau>0$. We deduce
\begin{align*}
&\norm{\T_{k}(Z_{k})-Z_{k}}^{2}_{2\bU-\bS}\geq \mu\norm{\T_{k}(Z_{k})-Z_{k}}^{2}\\
&=\mu\left[\norm{\T_{k}(Z_{k})-T(Z_{k})}^{2}+\norm{T(Z_{k})-Z_{k}}^{2}\right]\\
&-2\mu\inner{\T_{k}(Z_{k})-T(Z_{k}),Z_{k}-T(Z_{k})}\\
&\geq \tfrac{-\mu(1-\alpha)}{\alpha}\norm{\T_{k}(Z_{k})-T(Z_{k})}^{2}+(1-\alpha)\mu\norm{\T_{k}(Z_{k})-Z_{k}}^{2},
\end{align*}
where in the last inequality we have used the Fenchel-Young inequality.
%$
%\inner{\T_{k}(Z_{k})-T_{k}(Z_{k}),Z_{k}-T_{k}(Z_{k})}\leq \tfrac{1}{2\alpha}\norm{\T_{k}(Z_{k})-T_{k}(Z_{k})}^{2}+\tfrac{\alpha}{2}\norm{T_{k}(Z_{k})-Z_{k}}^{2}.
%$
%This implies 
%\begin{align*}
%\Ex_{k}\left[\norm{Z_{k+1}-z^{\ast}}^{2}_{\bS_{k+1}}\right]&\leq(1+\eta_{k})\norm{Z_{k}-z^{\ast}}^{2}_{\bS_{k}}+\mu_{k}\tfrac{1-\alpha}{\alpha}\Ex_{k}\left[\norm{\T_{k}(Z_{k})-T_{k}%(Z_{k})}^{2}\right]\\
%&-\mu_{k}(1-\alpha)\norm{T_{k}(Z_{k})-Z_{k}}^{2}+2(1+\eta_{k})\Ex_{k}[\nu_{k}]
%\end{align*}
A simple computation shows that 
\begin{align*}
&\norm{\T_{k}(Z_{k})-T_{k}(Z_{k})}^{2}=\norm{\Sigma_{k}L(\hat{X}_{k}-\bar{x}_{k})}^{2}+\norm{\hat{X}_{k}-\bar{x}_{k}}^{2}\\
&\leq \left(1+\cM(L^{\ast}\Sigma^{2}L)\right)\norm{\hat{X}_{k}-\bar{x}_{k}}^{2}\\
&\leq \tfrac{\left(1+\cM(L^{\ast}\Sigma^{2}L)\right)}{\cm(\Gamma^{-1})}\norm{\hat{X}_{k}-\bar{x}_{k}}^{2}_{\Gamma^{-1}}\\
&\leq \tfrac{\left(1+\cM(L^{\ast}\Sigma^{2}L)\right)}{\cm(\Gamma^{-1})}\norm{\scrF_{k}(X_{k})-\nabla f(X_{k})}^{2}_{\Gamma}.
\end{align*}
Hence, for $\alpha\in(0,1)$, we get  
\begin{align*}
&\Ex_{k}[\norm{\T_{k}(Z_{k})-Z_{k}}^{2}_{2\bU-\bS}]\geq (1-\alpha)\mu\norm{T_{k}(Z_{k}(\omega))-Z_{k}(\omega)}^{2}\\
&-\tfrac{\mu(1-\alpha)}{\alpha}\tfrac{\left(1+\cM(L^{\ast}\Sigma^{2}L)\right)}{\cm(\Gamma^{-1})}\Ex_{k}[\bar{\nu}_{k}](\omega).
\end{align*}
By continuity of the mappings $T(\cdot)$, it follows 
\[
0=\lim_{k\to\infty}\norm{T(Z_{k}(\omega))-Z_{k}(\omega)}^{2}=\norm{T(Z_\infty(\omega))-Z_\infty(\omega)}.
\]
Hence, $Z_\infty\in\Fix(T)=\Zer(\setG)$. As $\omega\in\Omega_{0}$ is arbitrary, the claim follows. 
\end{proof}

\section{Distributed Optimization}
\label{sec:distributed}
%----------------------------------------------------------------------
%%% Distributed Optimization
%----------------------------------------------------------------------
% !TEX root = ./STRIPD-CSS.tex

%\begin{algorithm*}[t]
%\caption{Distributed STriPD}
%\label{alg:SynchSTriPD}
%\SetAlgoLined
%Initialize: $x^{0}_{i}\in\scrH_{i},y^{0}_{i}\in\scrG_{i}$ for $i=1,\ldots,m$ and $w_{(i,j)}\in\scrW^{\ast}_{(i,j)}$ for all $(i,j)\in E$.\\
%Iteration $k$: For all agents $i=1,2,\ldots,m$
%\begin{align*}
%&\bar{w}_{(i,j),k}^{i}=\tfrac{1}{2}(w^{i}_{(i,j),k}+w^{j}_{(i,j),k})+\tfrac{\tau_{(i,j)}}{2}(A_{ij}x^{i}_{k}+A_{ji}x^{j}_{k}-b_{(i,j)})\quad\forall j\in\scrN_{i},\\
%&\bar{v}^{i}_{k}=\prox_{\sigma_{i}h^{\ast}_{i}}(v_{k}^{i}+\sigma_{i}L_{i}x^{i}_{k}),\\
%&x^{i}_{k+1}=\prox_{\gamma_{i}g^{\ast}_{i}}[x^{k}_{i}-\gamma \scrF^{i}_{k}(x^{k})-\gamma_{i}L^{\ast}_{i}\bar{v}_{i}^{k}-\gamma_{i}\textstyle{\sum_{j\in\scrN_{i}}}A_{ij}w^{i}_{(i,j),k}]\\
%&v^{i}_{k+1}=\bar{v}^{i}_{k}+\tau_{(i,j)}(x^{i}_{k+1}-x^{i}_{k})\\
%&w^{i}_{(i,j),k+1}=\bar{w}^{i}_{(i,j),k}+\tau_{i,j}A_{i,j}(x^{i}_{k+1}-x^{i}_{k})\quad\forall j\in\scrN_{i}.
%\end{align*}
%\end{algorithm*}

In this section we consider a network of agents, whose aim is to solve problem \eqref{eq:DistOpt} in a cooperative way. Consider an undirected graph $G=(V,E)$ over a vertex set $V=\{1,\ldots,m\}$ with edge set $E\subset V\times V$. Each vertex is associated with an agent, which is assumed to have a local memory and computational unit and can only communicate with its neighbors. We define the neighborhood of agent $i$ as $\scrN_{i}=\{j\in V\vert (i,j)\in E\}$. Each agent in the network is characterized by a private cost function 
$f_{i}(x_{i})+g_{i}(x_{i})+h_{i}(L_{i}x_{i}),$ defined over the vector space $\scrH_{i}\equiv\R^{p_{i}}$. Let $\scrH=\prod_{i=1}^{m}\scrH_{i}$. Furthermore, the decisions of the agents in the network are subject to affine constraints, coupling the decisions of neighboring agents. %Overall, we therefore aim to solve the distributed optimization problem \eqref{eq:DistOpt}.
% Let $\scrH\eqdef \prod_{i=1}^{m}\scrH_{i}$. The goal is to solve the following structured convex optimization problem in a distributed way: 
%\begin{equation}\label{eq:DistOpt}
%\begin{split}
%\min_{(x_{1},\ldots,x_{m})\in\scrH}\sum_{i=1}^{m}\Psi_{i}(x_{i})\\
%\text{s.t.: }A_{ij}x_{i}+A_{ji}x_{j}=b_{(i,j)}\quad\forall (i,j)\in E
%\end{split}
%\end{equation} 
 %%%%%
\begin{assumption} For each $i=1,\ldots,m$: 
\begin{itemize}
\item[(i)] For all $j\in\scrN_{i}$, $b_{(i,j)}\in\scrE_{(i,j)}\equiv\R^{l_{ij}}$ and $A_{ij}\in\R^{p_{i}\times l_{ij}}$;
\item[(ii)] $g_{i}:\R^{p_{i}}\to(-\infty,\infty],h_{i}:\R^{q_{i}}\to(-\infty,\infty]$ are proper closed convex and lower semi-continuous functions, and $L_{i}\in\R^{p_{i}\times q_{i}}$;
\item[(iii)] $f_{i}:\scrH_{i}\to\R$ is a convex, continuously differentiable and for some $\beta_{i}\geq 0$, $\nabla f_{i}$ is $\beta_{i}$-Lipschitz continuous in the norm $\norm{\cdot}_{\bQ_{i}}$;
\item[(iv)] The graph $G$ is connected;
\item[(v)] The set of solutions of \eqref{eq:DistOpt} is nonempty and there exists $x_{i}\in\rint(\dom g_{i})$ such that $L_{i}x_{i}\in \rint(\dom h_{i})$ and $A_{ij}x_{i}+A_{ji}x_{j}=b_{(i,j)}$ for $(i,j)\in E$. 
\end{itemize}
\end{assumption}

Let $\scrW_{(i,j)}\eqdef\R^{2l_{ij}}$ for all $(i,j)\in E$, with generic element $w_{ij}=[w_{(i,j)}^{i},w_{(i,j)}^{j}]$. The interpretation is that $w_{(i,j)}^{i}$ is controlled by agent $i$ and $w_{(i,j)}^{j}$ is controlled by agent $j$. For each edge $(i,j)\in E$ define the set
$$
\setK_{(i,j)}\eqdef\{(w_{1},w_{2})\in\scrW_{(i,j)}\vert w_{1}+w_{2}=b_{(i,j)}\},
$$
and the linear map $\setN_{(i,j)}:\scrH\to\R^{2l_{ij}}$ by 
$$
\inner{\setN_{(i,j)}x,w_{ij}}=\inner{A_{ij}x_{i},w_{ij}^{i}}+\inner{A_{ji}x_{j},w_{ij}^{j}}\qquad\forall x\in\scrH.
$$
Accordingly, we let $\setN:\scrH\to\prod_{(i,j)\in E}\R^{2l_{ij}}$ to be the operator defined by 
$
\inner{\setN x,w}=\sum_{(i,j)\in E}(\inner{A_{ij}x_{i},w_{(i,j)}^{i}}+\inner{A_{ji}x_{j},w_{(i,j)}^{j}}).$
%Note that the adjoint $\setN^{\ast}$ is characterized by 
%\[
%\inner{\setN x,w}=\inner{x,\setN^{\ast}w}=\sum_{(i,j)\in E}\left(\inner{x_{i},A^{\ast}_{ij}w_{(i,j)}^{i}}+\inner{x_{j},A^{\ast}_{ji}w_{(i,j)}^{j}}\right).
%\]
%That means, formally we can define the adjoint operators $\setN^{\ast}_{ij}:\scrW_{(i,j)}\to \scrH$ by 
%\[
%\inner{x,\setN^{\ast}w}=\sum_{i=1}^{m}\inner{x_{i},\sum_{j\in\scrN_{i}}A^{\ast}_{ij}w_{(i,j)}^{i}}.
%\]
Using these concepts, we can reformulate problem \eqref{eq:DistOpt} as 
\[
\min_{x_{1},\ldots,x_{m}}\sum_{i=1}^{m}[f_{i}(x_{i})+g_{i}(x_{i})+h_{i}(L_{i}x_{i})]+\sum_{(i,j)\in E}\delta_{\setK_{(i,j)}}(\setN_{(i,j)}x).
\]
Let $\setK\eqdef \prod_{(i,j)\in E}\setK_{(i,j)}$ and $L:\scrH\to\prod_{i}\R^{q_{i}}$ be defined by 
$Lx=[L_{1}x_{1},\ldots,L_{m}x_{m}]$. Set $Dx=[Lx,\setN x]$. Define the functions $F(x)\eqdef \sum_{i=1}^{m}f_{i}(x_{i})$, $G(x)\eqdef\sum_{i=1}^{m}g_{i}(x_{i})$ and $H(Dx)\eqdef h(Lx)+\delta_{\setK}(\setN x)$. With this notation, we have converted problem \eqref{eq:DistOpt} to problem \eqref{eq:Opt} with the functions $F,G$ and $H$. 

As in Section \ref{sec:algorithm}, the primal-dual optimality conditions can be written in the compact form as an inclusion problem involving the (maximally monotone)  operators
$$\begin{aligned}
&\setA(v,w,x)=[\partial h^{\ast}(v),\partial\delta_{\setC}^{\ast}(w),\partial g(x)],\\
&\setM(v,w,x)=[-Lx,-\setN x,L^{\ast}v+\setN^{\ast}w], \\
&\setC(v,w,x)=[0,0,\nabla f(x)].
\end{aligned}$$
The dual variable is the pair $y=[v,w]\in\scrG$ and we can apply Algorithm \ref{alg:STriPD} directly to solve the distributed optimization problem \eqref{eq:DistOpt}. The resulting stochastic process $\{Z_{k}\}_{k\geq 0}$ decomposes to agent-specific updates as described in Algorithm \ref{alg:SynchSTriPD} (cf. \cite{TriPD}). Given the identification of the operators characterizing the optimality conditions of a primal-dual pair, the convergence of the synchronous STriPD (Algorithm \ref{alg:SynchSTriPD}) follows immediately from Theorem \ref{thm:main}. 

\begin{figure*}[h!]
\centering
\begin{subfigure}[t]{0.3\textwidth}
\includegraphics[width=\columnwidth]{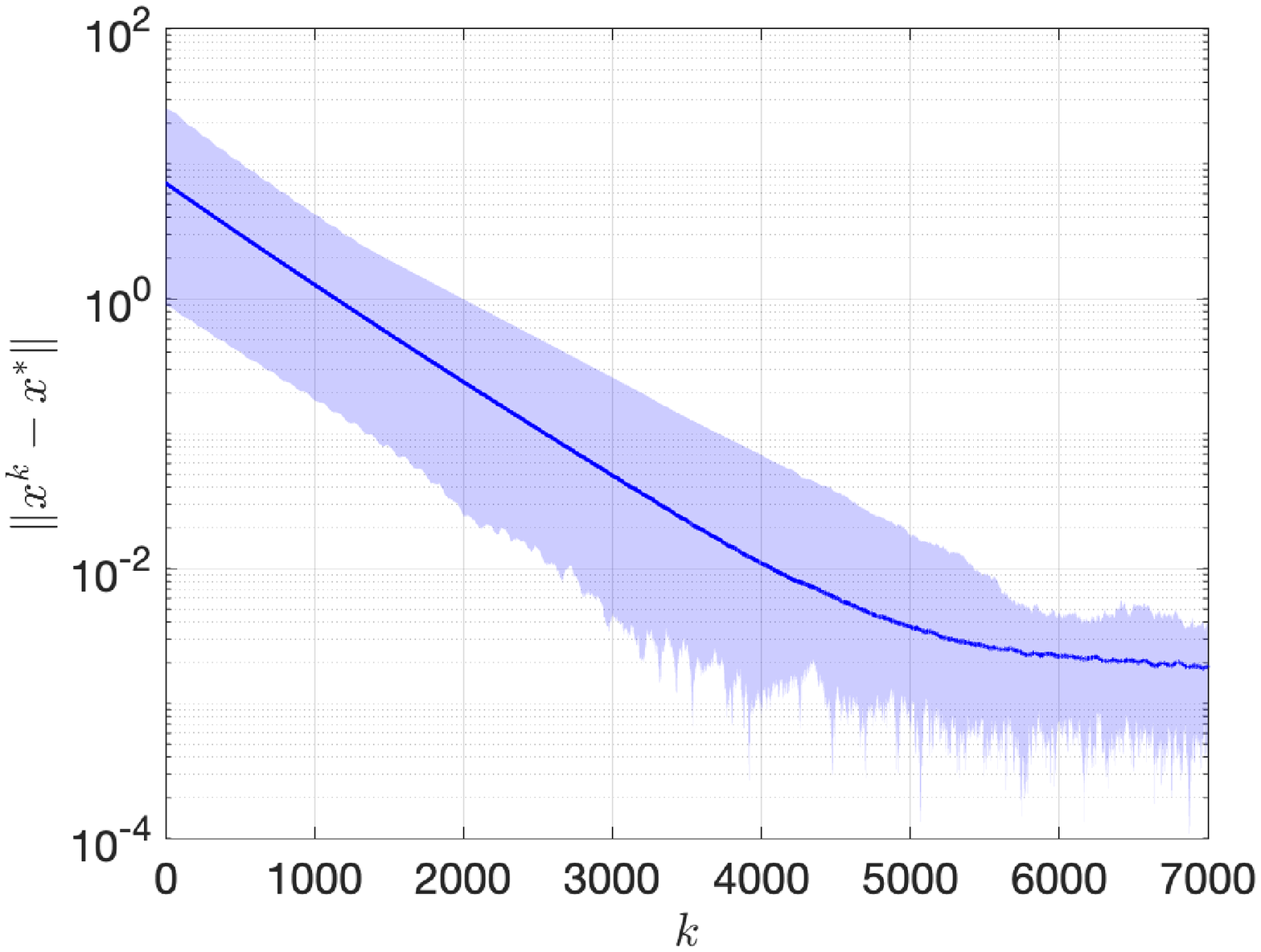}
\caption{Distance from the solution.}\label{fig_sol}
\end{subfigure}
\begin{subfigure}[t]{0.3\textwidth}
\includegraphics[width=\columnwidth]{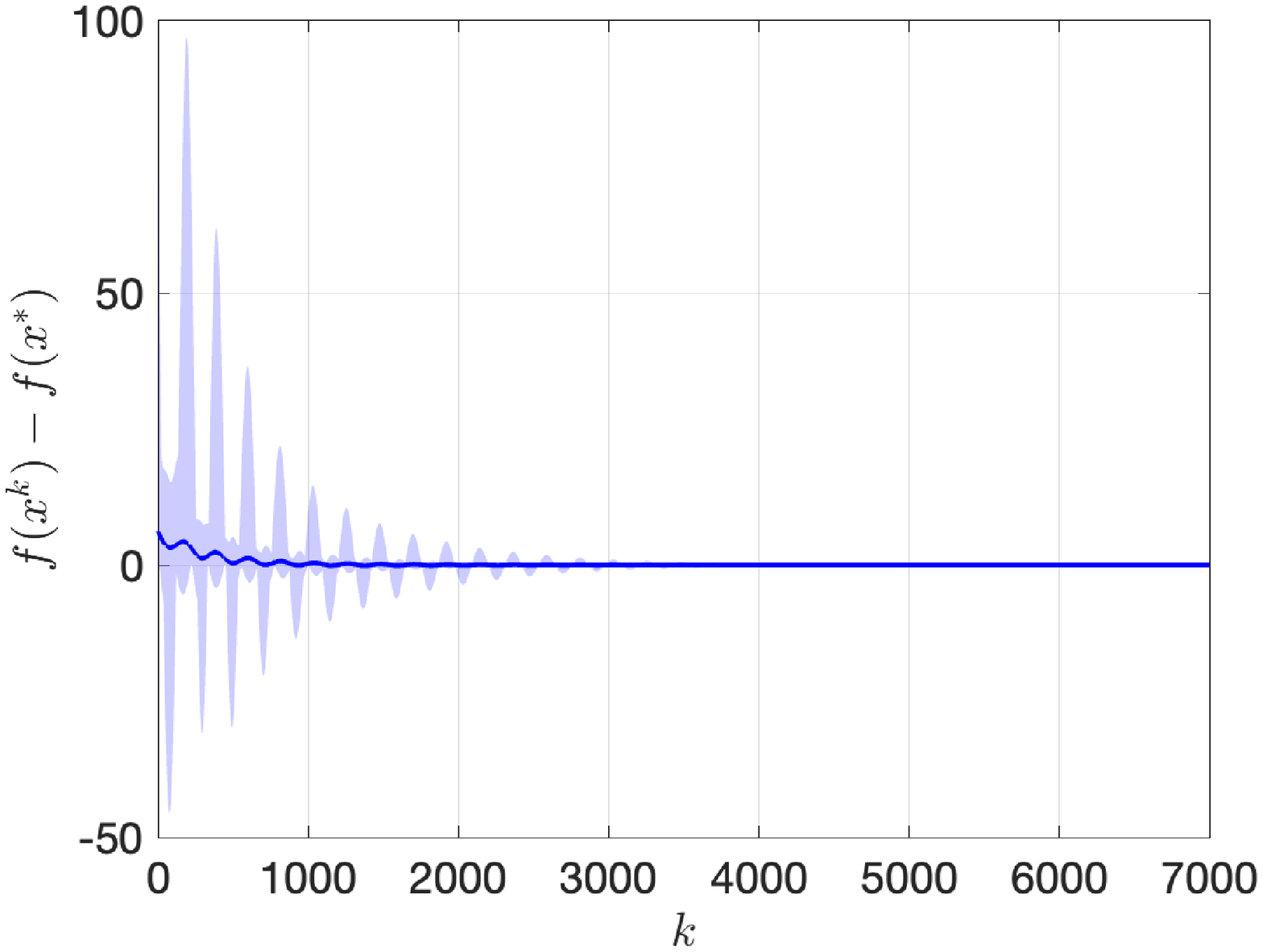}
\caption{Distance from minimum cost.}\label{fig_cost}
\end{subfigure}
\begin{subfigure}[t]{0.3\textwidth}
\includegraphics[width=\columnwidth]{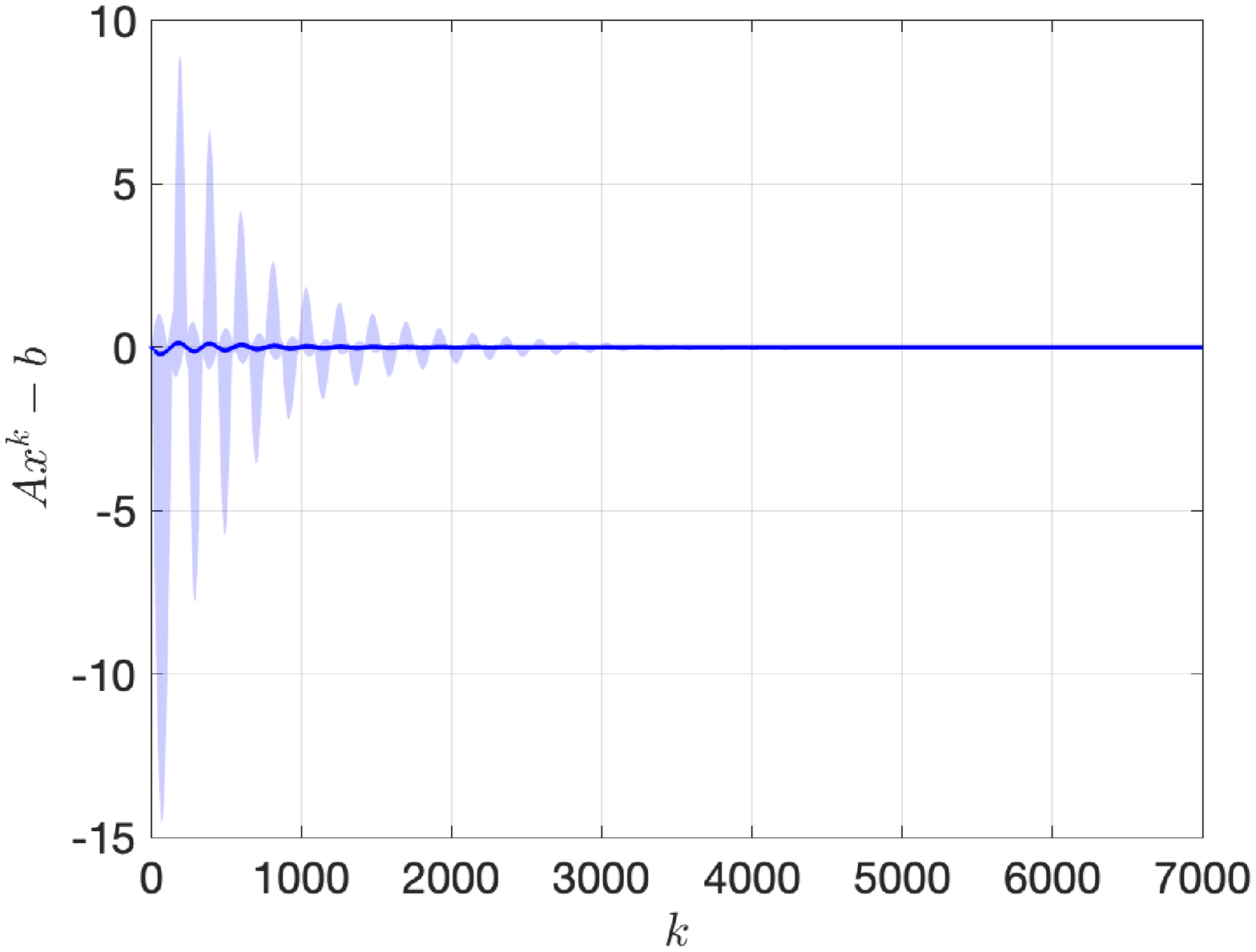}
\caption{Feasibility of the iterates.}\label{fig_constr}
\end{subfigure}
\end{figure*}

\section{Numerical Example}
%----------------------------------------------------------------------
%%% Numerics
%----------------------------------------------------------------------
% !TEX root = ./STRIPD-CSS.tex

%In this section we validate the theoretical analysis of the stochastic TriPD algorithm. Specifically, 

We test Algorithm \ref{alg:STriPD} on an economic dispatch problem for power grids, inspired by \cite{yi2016, li2021}. Consider $n$ control areas, each with a generator that supply power $x_i\in\R$ and a local demand $b_i\in \R_+$ that should be satisfied. Each generator has security bound of the form $\underline x_i\leq x_i\leq \bar x_i$, with $\underline x_i, \bar x_i\geq 0$ for all $i=1,\dots,m$. Each area has a local generation cost $f_i(x_i):\R\to\R$ so that the optimization problem is to minimize the overall cost $f(x)=\sum_{i=1}^m f_i(x_i)$, subject to the satisfaction of the demand:
\begin{equation}\label{eq_sim}
\begin{cases}
\min\limits_{x_i,i=1,\dots,m} & f(x)=\sum_{i=1}^m f_i(x_i)\\
\quad\; \text{s.t.} & \sum_{i=1}^m x_i=\sum_{i=1}^m b_i\\
 & \underline x_i\leq x_i\leq \bar x_i, i=1,\dots,m
\end{cases}
\end{equation}
%In words, each agent is responsible to decide the generation $x_i$ in its control area $i$ with the goal to minimize the global cost, while meeting the total demands constraint and its capacity bounds.
%
To write the problem in \eqref{eq_sim} in the form of problem \eqref{eq:DistOpt}, let us take $g_i(x_i)=\delta_{X_i}(x_i)$, i.e., the indicator function of the local constraints $X_i\eqdef[\underline x_i,\bar x_i]$, and $h(Lx)=\delta_C(Lx)$, where $L=\mathbf{1}_{m}$ and $C=\{y\in\R^m:y=b\}$ represents the coupling constraints.

Similarly to \cite{li2021}, we consider $m=5$ generators with cost functions $f_i(x_i)=\Ex[q_i(\xi)x_i^2+p_ix_i]$ where the coefficients are chosen with mean $q=[0.094, 0.078, 0.105, 0.082, 0.074]$ and $p=[1.22, 3.41, 2.53, 4.02, 3.17]$. The random variable is drawn according to a normal distribution with mean $q$. The local bounds on the supply power are $\underline x=[10,8,3.8,5.4,4.2]$ and $\bar x=[80 ,60,40,45,18]$ while the demands take the values $b=[35,20,25,30,10]$. 
%The solution can be computed exactly as  $x^*=[32.81,25.51,23.14,20.54,18]$ \cite{li2021}.
%\begin{figure*}[h!]
%\centering
%\begin{subfigure}[t]{0.3\textwidth}
%\includegraphics[width=\columnwidth]{Fig/dispatch_sol.eps}
%\caption{Distance from the solution.}\label{fig_sol}
%\end{subfigure}
%\begin{subfigure}[t]{0.3\textwidth}
%\includegraphics[width=\columnwidth]{Fig/dispatch_cost.eps}
%\caption{Decrease toward optimal cost.}\label{fig_cost}
%\end{subfigure}
%\begin{subfigure}[t]{0.3\textwidth}
%\includegraphics[width=\columnwidth]{Fig/dispatch_constr.eps}
%\caption{Feasibility of the iterates.}\label{fig_constr}
%\end{subfigure}
%\end{figure*}
We run the algorithm 100 times and plot in a thick blue line the average results; the transparent areas indicate the minimum and maximum values reached.
Figure \ref{fig_sol} displays the distance of the iterates from the solution and Fig \ref{fig_cost} illustrates the distance of the cost from the optimal value. In Figure \ref{fig_constr}, we show that asymptotically the constraints are satisfied. 

%
%
%
%\begin{figure}
%\includegraphics[width=0.4\textwidth]{Fig/dispatch_sol.eps}
%\label{fig:fig_sol}
%\caption{Distance from the solution.}
%\end{figure}
%\begin{figure}
%\includegraphics[width=0.4\textwidth]{Fig/dispatch_cost.eps}
%\label{fig_cost}
%\caption{Decrease toward optimal cost.}
%\end{figure}
%\begin{figure}
%\includegraphics[width=0.4\textwidth]{Fig/dispatch_constr.eps}
%\label{fig_constr}
%\caption{Feasibility of the iterates.}
%\end{figure}

\section{Conclusions}
\label{sec:conclusion}
%----------------------------------------------------------------------
%%% Conclusion
%----------------------------------------------------------------------
% !TEX root = ./STRIPD-CSS.tex
We propose a stochastic triangular preconditioned primal-dual algorithm (STriPD) for solving a large family of structured convex optimization problems, and multi-agent versions thereof. Many extensions of the present work will be investigated in a more elaborate investigation. Such extensions will include block-coordinate descent implementations, allowing for asynchronous updates in the distributed case. Furthermore, we will be investigating the iteration and oracle complexity of the method in detail.

%\addtolength{\textheight}{-12cm}   % This command serves to balance the column lengths
                                  % on the last page of the document manually. It shortens
                                  % the textheight of the last page by a suitable amount.
                                  % This command does not take effect until the next page
                                  % so it should come on the page before the last. Make
                                  % sure that you do not shorten the textheight too much.

%%%%%%%%%%%%%%%%%%%%%%%%%%%%%%%%%%%%%%%%%%%%%%%%%%%%%%%%%%%%%%%%%%%%%%%%%%%%%%%%

%%%%%%%%%%%%%%%%%%%%%%%%%%%%%%%%%%%%%%%%%%%%%%%%%%%%%%%%%%%%%%%%%%%%%%%%%%%%%%%%

%%%%%%%%%%%%%%%%%%%%%%%%%%%%%%%%%%%%%%%%%%%%%%%%%%%%%%%%%%%%%%%%%%%%%%%%%%%%%%%%
%\section*{APPENDIX}
%\label{sec:appendix}
%
%Appendixes should appear before the acknowledgment.
%\input{Appendix}

%%%%%%%%%%%%%%%%%%%%%%%%%%%%%%%%%%%%%%%%%%%%%%%%%%%%%%%%%%%%%%%%%%%%%%%%%%%%%%%%

% Generated by IEEEtran.bst, version: 1.14 (2015/08/26)

\end{document}